\documentclass[15pt]{article}

\usepackage{amsmath,amsfonts,amssymb,amsthm,amscd}

\newcounter{stepctr}
{\end{list}}

\newtheorem{thm}{Theorem}[section]
\newtheorem{prop}[thm]{Proposition}
\newtheorem{cor}[thm]{Corollary}

\theoremstyle{definition}
\newtheorem{dfn}[thm]{Definition}
\newtheorem{ex}[thm]{Example}

\newtheorem{rema}[thm]{Remark}
\newtheorem{lem}[thm]{Lemma}
\newtheorem{prob*}{Open problem}
\newcommand{\demo}{\begin{proof}}

\newcommand{\NN}{\mathcal{N}}
\newcommand{\R}{\ensuremath{\mathcal{R}}}

\newcommand{\N}{\mathbb{N}}

\def\ll^2{{\mathcal L}(\ell^2(\N))}

\def\f^0x{{\mathcal F^0}(X) }

\title
{\bf The Berkani's property and  a note on some recent results }

\author{  Zakariae Aznay,  Hassan  Zariouh}

\date{}
\begin{document}

\maketitle \thispagestyle{empty}

\begin{abstract}\noindent\baselineskip=10pt
 In this paper, we continue the  study of  property $(UW_{\Pi})$ introduced in \cite{berkani2},  in connection with other Weyl type theorems. Moreover,  we  give counterexamples to show that some   recent results related to this property, which are  announced and proved by P. Aiena and M. Kachad in \cite{aiena1} are false. Furthermore, we specify  the mistakes committed in each of them and  we give the  correct versions. We also give a global  note on the paper \cite{jayanthi}.

\end{abstract}

 \baselineskip=15pt
  \baselineskip=15pt
\section{Introduction}
 We continue the study of  properties introduced  in \cite{berkani1, berkani2}, in connection with other properties and Weyl type theorems.
 Moreover, we prove by counterexample (see Remark I), that   we do not expected  that an operator satisfying property $(UW_{\Pi})$ satisfies property $(Z_{\Pi_{a}}),$  contrarily to what has been proved in \cite[Thorem 2.5]{aiena1}. Furthermore, we give the correct version of \cite[Thorem 2.5]{aiena1} by proving  in Theorem \ref{thm1}, that if $T\in L(X)$ is an operator satisfying property $(UW_{\Pi})$ and $\Pi_{a}(T)\cap\sigma_{uw}(T)=\emptyset,$ then it satisfies property $(Z_{\Pi_{a}}).$ We also give the correct version of \cite[Thorem 2.6]{aiena1} (see Remark II and Theorem \ref{thm2}). As deduction, we give  analogous  result to Theorem \ref{thm2} for the properties $(UW_{E})$ and $(Z_{E_{a}}).$ And, we give in Remark IV a very crucial observation on the results of   the paper \cite{jayanthi}.
\section{Terminology and preliminaries}

Let $X$ denote an infinite dimensional complex Banach space, and
we denote by $L(X)$ the algebra of all bounded linear operators on $X.$
For $T\in L(X),$ we  denote by $\alpha(T)$ the dimension of the
kernel $\NN(T)$ and by $\beta(T)$ the codimension of the range $\R(T).$
By $\sigma (T)$ and $\sigma_a(T),$ we denote the spectrum,
 the approximate spectrum  of $T,$
respectively.
For an operator $T\in
L(X),$ the \textit{ascent} $p(T)$ and the \textit{descent} $q(T)$ of $T$
are defined by $p(T) = \inf\{n\in\N : \NN(T^n) = \NN(T^{n+1})\}$ and
$q(T)= \inf\{n\in\N :\R(T^n) = \R(T^{n+1})\},$ respectively; the
infimum over the empty set is taken $\infty.$  In order to simplify, and to give a global view to the reader, we use the same symbols and notations used in \cite{aiena1}. For more details on the several  classes and spectra originating from  Fredholm theory and B-Fredholm theory, we refer the reader to \cite{aiena1, berkani2, zariouh}.

In the following  list, we summarize  the same   notations and symbols used in \cite{aiena1}, which will be
needed in this paper.

\smallskip
 \noindent The symbol $\bigsqcup$
stands for disjoint union,\\
\noindent $\mbox{iso}\,A$: isolated points of a  subset $A\subset \mathbb{C},$\\
 \noindent $A^C$: the complementary of a subset $A\subset \mathbb{C}.$\\

 \begin{tabular}{l|l}
   $\sigma_{b}(T)$:  Browder spectrum of $T$ &  $\Delta^g(T):=\sigma(T)\setminus\sigma_{bw}(T)$\\
   $\sigma_{ub}(T)$:  upper semi-Browder spectrum of $T$ & $\Delta_{a}^g(T):=\sigma_{a}(T)\setminus\sigma_{ubw}(T)$\\
   $\sigma_{w}(T)$: Weyl spectrum of $T$ & $\Delta(T):=\sigma(T)\setminus\sigma_{w}(T)$\\
   $\sigma_{bw}(T)$: B-Weyl spectrum of $T$ & $\Pi^0(T)$: poles of $T$ of finite rank\\
   $\sigma_{uw}(T)$: upper semi-Weyl spectrum of $T$ & $\Pi_a(T)$: left  poles of $T$\\
   $\sigma_{ubw}(T)$: upper semi-B-Weyl spectrum of $T$ & $\Pi_a^0(T)$: left  poles of $T$ of finite rank\\
   $\sigma_{ld}(T)$: left Drazin spectrum of $T$ & $\Pi(T)$: poles of $T$\\
   $ \sigma_{d}(T)$:  Drazin spectrum of $T,$& $E_a(T):=\mbox{iso}\,\sigma_a(T)\cap\sigma_{p}(T)$\\
    $\sigma_{p}(T)$:  eigenvalues of $T$ & $E(T):=\mbox{iso}\,\sigma(T)\cap\sigma_{p}(T)$\\
$\sigma_{p}^0(T)$: eigenvalues of $T$ of finite multiplicity & $E^0(T):=\mbox{iso}\,\sigma(T)\cap\sigma_{p}^0(T)$\\
$\Delta_{a}(T):=\sigma_{a}(T)\setminus\sigma_{uw}(T)$ & $E_a^0(T):=\mbox{iso}\,\sigma_a(T)\cap\sigma_{p}^0(T)$
\\
\end{tabular}\\

The following property has relevant role in local
spectral theory: a bounded linear
 operator $T\in L(X)$ is said to have the {\it single-valued
 extension property} (SVEP for short) at $\lambda\in\mathbb{C}$ if
  for every open neighborhood $U_\lambda$ of $\lambda,$ the  function $f\equiv 0$ is the only
 analytic solution of the equation
 $(T-\mu I)f(\mu)=0\quad\forall\mu\in U_\lambda.$ We denote by
  ${\mathcal S}(T)=\{\lambda\in\mathbb{C}: T\mbox{  does not have the SVEP at } \lambda\}$
     and we say that  $T$ has  SVEP
  if
$ {\mathcal S}(T)=\emptyset.$ We say that $T$ has the SVEP on $A\subset\mathbb{C},$ if $T$ has the SVEP
at every $\lambda\in A.$ (For more details about this property, we refer the reader to \cite{Aiena1}). Thus it follows easily  that $T\in L(X)$ has the  SVEP at every point of the boundary $\partial\sigma(T)$ of the spectrum $\sigma(T).$ In particular, $T$ has the SVEP at every  isolated  point in
$\sigma(T).$ We also have from  \cite[Theorem 2.65]{Aiena1},  \[p(T-\lambda_0 I)<\infty \Longrightarrow
\mbox{ T has the SVEP at } \lambda_0, \,\,\,(A)\] and dually
\[q(T-\lambda_0 I)<\infty \Longrightarrow \mbox{ $T^*$ has the SVEP
at } \lambda_0,\,\,\,(B),\]
 Furthermore, if $T-\lambda_0 I$ is
semi-B-Fredholm then the implications above are equivalences.

 Note that it well known that  if $T^*$ has the SVEP on $(\sigma_{ubw}(T))^C,$ then $\sigma_{a}(T)=\sigma(T),$
$\sigma_{w}(T)=\sigma_{uw}(T),$ $\sigma_{bw}(T)=\sigma_{ubw}(T),$ $E(T)=E_{a}(T)$ and $\Pi(T)=\Pi_{a}(T).$ And this gives a short and simple proof of
\cite[Theorem 2.15]{aiena2}.
\section{Correct versions and addendums}

\begin{dfn} \cite{berkani1, berkani2} \label{dfn1}Let $T\in L(X).$ $T$ is said to satisfy\\
a) property $(UW_{\Pi})$  if $\Delta_{a}(T)=\Pi(T).$\\
b) property $(UW_{E})$  if $\Delta_{a}(T)=E(T).$\\
c) property $(UW_{E_a})$  if $\Delta_{a}(T)=E_a(T).$
\end{dfn}

 \noindent We begin this section by the following lemma, which will be   useful in the sequel.

\begin{lem}\label{lem1}let $T\in L(X).$\\
(i) If $T$ satisfies property $(UW_{E_a}),$ then

$E_a(T)=E_a^0(T)=\Pi_a^0(T)=\Pi_a(T)\mbox{ and }   E(T)=E^0(T)=\Pi^0(T)=\Pi(T).$\\
(ii) If  $T$ satisfies property $(UW_{E}),$ then $E(T)=E^0(T)=\Pi_a^0(T)=  \Pi^0(T)=\Pi(T).$\\
(iii) If $T$ satisfies  property $(UW_{\Pi}),$ then $\Pi^0(T)=\Pi(T)=\Pi_a^0(T).$
\end{lem}

\begin{proof} (i) See, \cite[Remark 3.4, Theorem 3.5, Remark 2.5]{berkani1}.\\
(ii) Since $T$ satisfies property $(UW_{E}),$ then
$ \lambda\in E(T)\Longleftrightarrow \lambda\in\mbox{iso}\,\sigma_a(T)\cap(\sigma_{uw}(T))^C\Longleftrightarrow  \lambda\in \Pi_a^0(T).$ So
$E(T)=E^0(T)=\Pi_a^0(T)$ and $\Pi^0(T)=\Pi(T).$ We show $E^0(T)=\Pi^0(T)$ and  let $\lambda \in E^0(T)$ be arbitrary. Then $T$ and $T^*$ have the SVEP at $\lambda.$ As  $\lambda \in \mbox{iso}\,\sigma(T)\cap(\sigma_{uw}(T))^C.$ Hence $\lambda \in \Pi^0(T)$ and consequently, $E^0(T)=\Pi^0(T).$
As conclusion, we have $E(T)=E^0(T)=\Pi_a^0(T)=  \Pi^0(T)=\Pi(T).$\\
(iii) Goes similarly with (ii).
\end{proof}

From Lemma \ref{lem1}, we obtain immediately the following corollary.

\begin{cor}\label{cor1} Let $T\in L(X).$ Then $T$ satisfies property $(UW_{E})$ if and only if $T$ satisfies property $(UW_{\Pi})$ and $E(T)=\Pi(T).$
\end{cor}

By Corollary \ref{cor1}, if $T$ satisfies property $(UW_{E}),$ then it satisfies property $(UW_{\Pi}).$ But the converse is generally not true, as we can see in the following example.

\begin{ex} Hereafter, the Hilbert space $ l^2(\mathbb{N})$ is  denoted by $\l^2.$ We consider the operator $T$ defined on $l^2$ by $T(x_1,x_2,x_3,\ldots)=(\frac{x_2}{2},\frac{x_3}{3}, \frac{x_4}{4},\ldots).$ It easily seen that  $\sigma_a(T)=\sigma_{uw}(T)=\{0\}=E(T).$ Moreover, $\Pi(T)=\emptyset,$ since $p(T)=\infty.$ So $T$ satisfies property $(UW_{\Pi}),$ but it does not satisfy property $(UW_{E}).$ Note here that  $\Pi(T)\neq E(T).$
\end{ex}

\begin{rema}~~\label{rema1} \begin{enumerate}
  \item[1.] The properties $(UW_{E_a})$ and $(UW_{\Pi})$ are independent. 
  We consider the operator $T$ defined on the Banach space  $l^2\oplus l^2$ by $T=R\oplus P,$ where $R$ is the right shift operator defined on $l^2$ by $ R(x_1,x_2,x_3,\ldots)=(0,x_1,x_2,x_3,\ldots)$ and $P$ is the operator defined on $l^2$ by $ P(x_1,x_2,x_3,\ldots)=(0,x_2,x_3,x_4,\ldots).$ Then  $\sigma_a(T)=C(0, 1)\cup\{0\},$ $\sigma_{uw}(T)=C(0, 1);$ where $C(0, 1)$ is the unit circle
of $ \mathbb{C},$ $E_a(T)=\{0\}$ and $\Pi(T)=E(T)=\emptyset.$ Thus $T$ satisfies property  $(UW_{E_a}),$ but it does not satisfy property $(UW_{\Pi})$ and hence it does not satisfy property $(UW_{E})$ too. On the other hand, the operator $T$ defined on $l^2\oplus l^2$ by $T=R\oplus 0$ satisfies property $(UW_{\Pi}),$ but it does not satisfy property  $(UW_{E_a}).$ Indeed $\sigma_a(T)=\sigma_{uw}(T)=C(0, 1)\cup\{0\},$ $\Pi(T)=E(T)=\emptyset$ and $E_a(T)=\{0\}.$
\item[2.] The same examples defined  above in this remark, show that the properties  $(UW_{E_a})$ and $(UW_{E})$ are  independent (see also \cite{berkani2}).
\end{enumerate} \end{rema}

 Now we  prove in Remark I, Remark II and Remark III bellow,  by counterexamples  that   some results announced and proved by P. Aiena, M. Kachad in {\cite{aiena1} are false. Moreover, we  specify  the mistakes committed in each of them and  we give there  correct versions.

We recall that a bounded linear operator $T\in L(X)$ is said to satisfy  property $(Z_{\Pi_{a}})$  if  $\Delta(T)= \Pi_{a}(T),$ see \cite{zariouh}. \\
  {\bf \underline{Remark I}:} It is proved in \cite[Theorem 2.5]{aiena1} that if $T\in L(X)$ satisfies property  $(UW_{\Pi}),$ then it satisfies property $(Z_{\Pi_{a}}).$ But this result is false. Indeed, if we  consider  the operator $T=R\oplus 0$ defined above in Remark \ref{rema1}, then  $\sigma(T)=\sigma_{w}(T)=D(0, 1);$ where $D(0, 1)$ is the  closed unit disc in $ \mathbb{C}.$ Furthermore, it is easily seen that $\sigma_{ubw}(T)=\sigma_{ld}(T)=C(0, 1)$ and so $\Pi_a(T)=\{0\}.$  Hence $T$ satisfies property   $(UW_{\Pi}),$ since $\Delta_{a}(T)=\emptyset =\Pi(T),$ but it does not satisfy property  $(Z_{\Pi_{a}}),$ since $\Delta(T)=\emptyset \neq \Pi_a(T).$

  The mistake made in the   proof of \cite[Theorem 2.5]{aiena1} is as follows: the equality $$\sigma(T)=(\sigma_{w}(T)\setminus\sigma_{uw}(T))\bigsqcup\Pi_a(T)\bigsqcup\sigma_{ubw}(T)$$ proved in line  8 of the proof of   \cite[Theorem 2.5]{aiena1} gave the following incorrect  decision $(I_1)$ (see line 9 and 10 of the proof of   \cite[Theorem 2.5]{aiena1}): \\Since $\sigma_{ubw}(T)\subset\sigma_{uw}(T)$ then $$ \sigma(T)\subset(\sigma_{w}(T)\setminus\sigma_{uw}(T))\bigsqcup\Pi_a(T)\bigsqcup\sigma_{uw}(T)=\sigma_{w}(T)\bigsqcup\Pi_a(T).\,\,\,\,\,\, (I_1)$$

  And the inclusion $(I_1)$ gave the incorrect equality $(I_2)$ bellow (see line 11 of  the proof of  \cite[Theorem 2.5]{aiena1}) $$\sigma(T)=\sigma_{w}(T)\bigsqcup\Pi_a(T).\,\,\,\,\,\, (I_2)$$

  But here we don't have always the following  disjoint union $(\sigma_{w}(T)\setminus\sigma_{uw}(T))\bigsqcup\Pi_a(T)\bigsqcup\sigma_{uw}(T)$ [even if $T$ satisfies property $(UW_{\Pi})$], since
$\Pi_a(T)\cap\sigma_{uw}(T)$ is not necessarily empty. Indeed, It already mentioned that the operator $T=R\oplus 0$ defined above satisfies property $(UW_{\Pi}).$ But since  $\sigma_{uw}(T)=C(0, 1)\cup\{0\}$ and $\Pi_a(T)=\{0\},$ then $\sigma_{uw}(T)\cap\Pi_a(T)=\{0\}\neq\emptyset.$ Moreover, this example shows that $\sigma_{w}(T)\cap\Pi_a(T)=\{0\}\neq\emptyset.$ Thus we cannot consider the union $\sigma_{w}(T)\cup\Pi_a(T)$ as a disjoint union considered   in the second member of the inclusion  $(I_1).$ Consequently, $\sigma_{w}(T)\cap\Pi_a(T)\neq\emptyset$  contrarily to what has been proved in the equality $(I_2)$ above. Note that here we have $\sigma(T)\setminus\sigma_{w}(T)=\emptyset\neq \Pi_a(T)=\{0\}.$

We recall that if we have two sets $A$ and $B$ of $\mathbb{C}$ such that there union is disjoint and $B$ is a subset of $C,$ then we cannot expected  generally that the union $A\cup C$ will be  disjoint. And this was exactly the origin of the error made in the inclusion $(I_1).$\\

Among other results, we prove in the next proposition that the properties $(UW_{\Pi})$ and $(Z_{\Pi_{a}})$ are independent.

\begin{prop} \label{prop1} The properties $(UW_{\Pi})$ and $(Z_{\Pi_{a}})$ are independent.
\end{prop}

\begin{proof} It already mentioned in Remark I,  that the operator $T=R\oplus 0$ satisfies  property  $(UW_{\Pi}),$ but it  does not satisfy property  $(Z_{\Pi_{a}}).$ On the other hand, if we consider  the operator $T=R\oplus R\oplus L;$ where $L$ is the left shift operator. Then $\sigma(T)=\sigma_{a}(T)=\sigma_{w}(T)=D(0, 1)$ and so $\Delta(T)=\Pi(T)=\Pi_{a}(T)=\emptyset.$ Thus $T$ satisfies  property $(Z_{\Pi_{a}}).$ Moreover, we have  $\alpha(T)=1$ and $\beta(T)=2,$ then $0\in \Delta_{a}(T).$ Hence $T$ does not satisfy  property $(UW_{\Pi}).$
\end{proof}
We give in  the following theorem the correct versions  of \cite[Theorem 2.5]{aiena1}.

\begin{thm}\label{thm1} Let $T\in L(X)$ be an operator satisfying property $(UW_{\Pi}).$ Then\\
 If $\Pi_{a}(T)\cap\sigma_{uw}(T)=\emptyset,$ then $T$ satisfies property $(Z_{\Pi_{a}}).$ In particular, property $(Z_{\Pi_{a}})$ holds for $T$ if $\sigma_{uw}(T)=\sigma_{ubw}(T).$
\end{thm}

\begin{proof}  It is clear that  $\Delta(T)\subset\Delta_{a}(T).$ By  hypothesis, $\Pi_{a}(T)\cap\sigma_{uw}(T)=\emptyset,$ it follows
 that $\Pi_{a}(T)=\Pi_{a}(T)\cap(\sigma_{uw}(T))^C\subset\Delta_{a}(T).$ As  $T$ satisfies property $(UW_{\Pi})$ then from Lemma \ref{lem1},   $\Delta_{a}(T)=\Pi^0(T).$  The inclusion $\Pi^0(T)\subset\Delta(T)$ is always true. Hence  $\Delta(T)=\Pi_{a}(T)$ and $T$ satisfies $(Z_{\Pi_{a}}).$

 In particular, if $\sigma_{uw}(T)=\sigma_{ubw}(T),$ then $\Pi_{a}(T)\cap\sigma_{uw}(T)=\Pi_{a}(T)\cap\sigma_{ubw}(T)=\emptyset.$ So property $(Z_{\Pi_{a}})$ holds for $T.$
\end{proof}

 Following \cite{zariouh}, we say that an operator $T\in L(X)$ satisfies property $(Z_{E_{a}})$ if $\Delta(T)=E_{a}(T).$ Similarly to Theorem \ref{thm1},  and to avoid redundancies, we give the following result without proof. Note  that the examples given in Proposition \ref{prop1} shows that the properties $(UW_{E})$ and $(Z_{E_{a}})$ are independent.

\begin{thm}\label{thm0} Let $T\in L(X).$ The  following statement holds.\\
 If $T$ satisfies property $(UW_{E})$ and $E_{a}(T)\cap\sigma_{uw}(T)=\emptyset,$ then $T$ satisfies property $(Z_{E_{a}}).$
\end{thm}

In the first assertion of the next theorem, we extend with simple and short proof
  \cite[Theorem 2.15]{aiena2} in which it is proved  that if $T^*$ has the SVEP, then the properties $(UW_{\Pi})$ and $(Z_{\Pi_{a}})$ for $T$ are equivalent. Note that property $(gb)$ is equivalent to say that $T^*$ has the SVEP on $\Delta_{a}^g(T).$ Recall that according to \cite{berkani-zariouh1}, an operator $T\in L(X)$ is said to satisfy property $(gb)$ if $\Delta_{a}^g(T)=\Pi(T).$

\begin{thm} Let $T\in L(X).$ We have the following statements.\\
(i) If $T$ satisfies property $(gb),$ then the properties $(UW_{\Pi})$ and $(Z_{\Pi_{a}})$ for $T$ are equivalent.\\
(ii) If $T$ satisfies the properties $(UW_{\Pi})$ and $(Z_{\Pi_{a}}),$ then  $T$ satisfies property $(gb).$\\
(iii) If $T\in L(X)$ satisfies the property $(Z_{\Pi_a}),$ then $\Pi(T)\subset \Delta_a(T).$
\end{thm}

\begin{proof}
(i) Suppose that $T$ satisfies property $(UW_{\Pi}).$ The property $(gb)$ for $T$ entails  from \cite[Corollary 2.9]{berkani-zariouh1} and Lemma \ref{lem1} that $\Delta(T)=\Pi^0(T)=\Pi(T)=\Pi_a(T).$ So $T$ satisfies property $(Z_{\Pi_{a}}).$ Conversely, if $T$ satisfies property $(Z_{\Pi_{a}}),$ then from \cite[Theorem 2.3]{berkani-zariouh1} and \cite[Lemma 2.9]{zariouh} that $\Delta_{a}(T)=\Pi^0(T)=\Pi(T).$\\
(ii) The properties  $(UW_{\Pi})$ and $(Z_{\Pi_{a}})$ imply from \cite[Lemma 2.9]{zariouh}, that $\Delta_{a}(T)=\Pi^0(T)$ and  $\Pi(T)=\Pi_a(T).$ This is equivalent from \cite[Theorem 2.10]{berkani-zariouh1} to say that $T$ satisfies property $(gb).$\\
 (iii) Since $T$ satisfies the property $(Z_{\Pi_a}),$ then $\Pi(T)=\Pi_a(T)=\Delta(T).$ As $\Delta(T)\subset\Delta_a(T),$ it follows that  $\Pi(T)\subset \Delta_a(T).$\end{proof}

\begin{rema}
\item[1.] Remark that all of the properties $(gb),$ $(UW_{\Pi})$ and $(Z_{\Pi_{a}})$ are mutually independent.
\begin{enumerate}
 \item[i.] We consider the operator $T$ defined on $l^2$ by $T(x_1,x_2,x_3,\ldots)=(x_2,0,0,\ldots).$ Then $\sigma(T)=\sigma_{a}(T)=\sigma_{w}(T)=\sigma_{uw}(T)=\Pi_{a}(T)=\Pi(T)=\{0\}$ and $\sigma_{ubw}(T)=\emptyset.$ So $T$ satisfies property $(gb),$ but it does not satisfy neither  property $(UW_{\Pi})$ nor property  $(Z_{\Pi_{a}}).$
\item[ii.] The operator $T=R\oplus R\oplus L$ defined in Proposition \ref{prop1} satisfies property $(Z_{\Pi_{a}}),$  but it  does not satisfy property $(gb),$ since $0\in \Delta_{a}^g(T).$
\item[iii.] The operator $T=R\oplus 0$ defined above in Remark \ref{rema1} satisfies property $(UW_{\Pi}),$   but it  does not satisfy property $(gb).$
\end{enumerate}

\item[2.] Remark that all of the properties $(gb),$ $(UW_{E}),$ $(UW_{E_{a}})$ and  $(Z_{\Pi_{a}})$ are mutually independent.
\begin{enumerate}
     \item[i.] We consider the operator $T$ defined on $l^2$ by $T(x_1,x_2,x_3,\ldots)=(x_2,0,0,\ldots).$ Then $T$ satisfies property $(gb),$ but it does not satisfy neither  property $(UW_{E})$ nor property  $(UW_{E_{a}}).$
\item[ii.] The operator $T=R\oplus 0$ defined above satisfies property $(UW_{E}),$  but it  does not satisfy neither property $(gb)$ nor property  $(Z_{\Pi_{a}}).$
\item[iii.] The operator $T=R\oplus R\oplus L$ defined above  satisfies property  $(Z_{\Pi_{a}}),$   but it  does not satisfy neither property $(UW_{E})$ nor property $(UW_{E_{a}}).$  And the operator $T=R\oplus P$ defined in Remark \ref{rema1}, satisfies property  $(UW_{E_{a}}),$ but it  does not satisfy neither property $(gb)$ nor property  $(Z_{\Pi_{a}}).$
\end{enumerate}
\end{rema}

 \noindent {\bf \underline{Remark II}:} P. Aiena and M. Kachad announced and proved in \cite[Theorem 2.6]{aiena1} the following result.

  \noindent Theorem: \cite[Theorem 2.6]{aiena1} Let $T \in L(X).$ Then the following statements are equivalent:\\
 (i) $T$  satisfies property $(UW_{\Pi});$\\
 (ii) $T$  satisfies property $(Z_{\Pi_a})$ and $\sigma_{w}(T)\setminus\sigma_{uw}(T) =\sigma(T)\setminus\sigma_{a}(T);$\\
 (iii) $T$  satisfies property $(Z_{\Pi_a})$ and $\Delta_{a}(T)\cap\sigma_{w}(T)=\emptyset.$\\

But this theorem is not true. Indeed, by Proposition \ref{prop1}, it follows that the implications $``(i)\Longrightarrow (ii)"$ and  $``(i)\Longrightarrow (iii)"$ are false. Note that the reasoning they used  for proving \cite[Theorem 2.6]{aiena1} is as follows: they proved that $``(i)\Longleftrightarrow (ii)"$ and $``(i)\Longleftrightarrow (iii)"$ and  they deduced by transitivity that $``(ii)\Longleftrightarrow (iii)".$ But, since   the properties $(UW_{\Pi})$ and $(Z_{\Pi_{a}})$ are independent (see Proposition \ref{prop1}), then  the proof of  $``(ii)\Longleftrightarrow (iii)"$ remains incorrect.

Moreover, in line   12 and 13 and 14 of the proof of  \cite[Theorem 2.6]{aiena1},   they said
$$`` \text{If } \lambda\in \Delta_{a}(T) \text{ then }\lambda \notin \sigma_{w}(T) \text{ and hence }\lambda \notin \sigma_{uw}(T). \text{  This implies that  } \lambda\in \Delta_{a}(T)=\Pi_{a}(T)=\Pi(T)"$$
But this is not clear.

Now, we give in Theorem \ref{thm2} bellow, the correct version of \cite[Theorem 2.6]{aiena1} in which we also  prove that the   equivalence $``(ii)\Longleftrightarrow (iii)"$ of \cite[Theorem 2.6]{aiena1}  is  true.  Before that, we give the following lemma.

\begin{lem}\label{lem2}
The following equivalence holds for every $T\in L(X).$\\
$$\Delta_{a}(T)\cap\sigma_{w}(T)=\emptyset \Longleftrightarrow \sigma_{w}(T)\backslash\sigma_{uw}(T)=\sigma(T)\backslash\sigma_{a}(T).$$
\end{lem}
\begin{proof}
$\Longrightarrow$) Suppose that $\Delta_{a}(T)\cap\sigma_{w}(T)=\emptyset$, then $(\sigma_{w}(T)\backslash\sigma_{uw}(T))\cap\sigma_{a}(T)=\emptyset.$ If $\lambda \in \sigma_{w}(T)\backslash\sigma_{uw}(T),$ then $\lambda \notin  \sigma_{a}(T)$, so that  $\lambda \in \sigma(T)\backslash\sigma_{a}(T).$ Thus $\sigma_{w}(T)\setminus\sigma_{uw}(T)\subset\sigma(T)\setminus\sigma_{a}(T).$ Conversely, it easily seen that without condition on $T,$  we have always that  $\sigma(T)\setminus\sigma_{a}(T)=\sigma_{w}(T)\cap (\sigma_{a}(T))^C,$ and so    $\sigma(T)\setminus\sigma_{a}(T)\subset \sigma_{w}(T)\setminus\sigma_{uw}(T)$. Hence,
  $\sigma_{w}(T)\setminus\sigma_{uw}(T)=\sigma(T)\setminus\sigma_{a}(T)$.\\
$\Longleftarrow$) Assume that   $\sigma_{w}(T)\setminus\sigma_{uw}(T)=\sigma(T)\backslash\sigma_{a}(T)$, then $\sigma_{w}(T)= \sigma_{uw}(T) \bigsqcup (\sigma(T)\setminus\sigma_{a}(T))$. Hence
 \begin{eqnarray*}
\Delta_{a}(T)\cap\sigma_{w}(T)&=& \Delta_{a}(T)\cap[\sigma_{uw}(T) \bigsqcup (\sigma(T)\setminus\sigma_{a}(T))]\\ &=& [\Delta_a(T)\cap\sigma_{uw}(T)]\bigsqcup[\Delta_{a}(T)\cap (\sigma(T)\setminus\sigma_{a}(T))]\\ &=& [\Delta_{a}(T)\cap (\sigma(T)\setminus\sigma_{a}(T))]\subset\sigma_{a}(T)\cap (\sigma(T)\setminus\sigma_{a}(T))=\emptyset. \end{eqnarray*}
\end{proof}

\begin{thm}\label{thm2}Let $T \in L(X).$ Then the following statements are equivalent:\\
 (i) $T$  satisfies property $(UW_{\Pi})$ and $\Pi_{a}(T)\cap\sigma_{uw}(T)=\emptyset;$\\
 (ii) $T$  satisfies property $(Z_{\Pi_a})$ and $\sigma_{w}(T)\setminus\sigma_{uw}(T) =\sigma(T)\setminus\sigma_{a}(T);$\\
 (iii) $T$  satisfies property $(Z_{\Pi_a})$ and $\Delta_{a}(T)\cap\sigma_{w}(T)=\emptyset.$
\end{thm}

\begin{proof} $(ii)\Longleftrightarrow(iii)$ Is an immediate consequence of   Lemma \ref{lem2}.\\
$(i)\Longrightarrow(ii)$  Follows directly from Theorem \ref{thm1} and \cite[Theorem 3.5]{berkani2}.\\
$(iii)\Longrightarrow(i)$ Since $T$ satisfies property $(Z_{\Pi_a})$  then $\Pi_{a}^0(T)=\Pi_{a}(T),$ see \cite[Lemma 2.9]{zariouh}. As $\Pi_{a}^0(T)\cap\sigma_{uw}(T)\subset\Delta_{a}(T)\cap\sigma_{w}(T),$ then $\Pi_{a}(T)\cap\sigma_{uw}(T)=\emptyset.$ Let us to prove that property  $(UW_{\Pi})$ holds for $T.$ Since $\Delta_{a}(T)\cap\sigma_{w}(T)=\emptyset,$  then $\Delta_{a}(T)=\Delta_{a}(T)\cap (\sigma_{w}(T))^C=\Delta(T).$ As $T$ satisfies property $(Z_{\Pi_a}),$ then $\Delta_{a}(T)=\Pi(T).$
\end{proof}

\begin{rema}

 If $T\in L(X)$  satisfies  property $(Z_{\Pi_a})$ and the a-Browder's theorem, then $T$ satisfies the property $(UW_{\Pi})$. Indeed,
 since $T$  satisfies the property $(Z_{\Pi_a}),$  then $\Pi_{a}^0(T)=\Pi(T).$   a-Browder's theorem for $T$ entails that $\Pi_a^0(T)=\Delta_a(T).$ Hence  $\Delta_a(T)=\Pi(T).$ So $T$ satisfies the property $(UW_{\Pi}).$ Remark that if $T$   satisfies property $(Z_{\Pi_a})$ and  $\Delta_{a}(T)\cap\sigma_{w}(T)=\emptyset$ then $T$ satisfies a-Browder's theorem.\end{rema}

Now, we give a similar result to Theorem \ref{thm2} for the property $(Z_{E_a}).$

\begin{thm}Let $T \in L(X).$ Then the following statements are equivalent:\\
 (i) $T$  satisfies property $(UW_{E})$ and $E_{a}(T)\cap\sigma_{uw}(T)=\emptyset;$\\
 (ii) $T$  satisfies property $(Z_{E_a})$ and $\sigma_{w}(T)\setminus\sigma_{uw}(T) =\sigma(T)\setminus\sigma_{a}(T);$\\
 (iii) $T$  satisfies property $(Z_{E_a})$ and $\Delta_{a}(T)\cap\sigma_{w}(T)=\emptyset.$
\end{thm}

\begin{proof} $(ii)\Longleftrightarrow(iii)$ Is already done above.\\
$(i)\Longrightarrow(ii)$  Follows directly from Corollary \ref{cor1} and Theorem \ref{thm0}.\\
$(iii)\Longrightarrow(i)$ Since $T$ satisfies property $(Z_{E_a})$  then $T$ satisfies property $(Z_{\Pi_a})$ and  $E_{a}(T)=\Pi_a(T),$ see \cite[Corollary 2.5]{zariouh}. It follows from Theorem \ref{thm2}, that $T$  satisfies property $(UW_{\Pi})$ and $E_{a}(T)\cap\sigma_{uw}(T)=\emptyset,$ As $E(T)=\Pi(T)$ see \cite[Lemma 2.3]{zariouh}. Then by Corollary \ref{cor1}, $T$  satisfies property $(UW_{E})$ and $E_{a}(T)\cap\sigma_{uw}(T)=\emptyset;$
\end{proof}

 \noindent {\bf \underline{Remark III}:} It is proved in  \cite[page 37]{aiena1}   that if $T\in L(X)$ is a finite-isoloid operator,  then $\sigma_{b}(T) = \sigma_{d}(T).$ Its proof is based on the fact that  \\
  $$\mbox{ if } \lambda \notin \sigma_{d}(T), \mbox{ then } \lambda \in \mbox {iso}\,\sigma(T).\,\,\,\,\,\, (I_3)$$ But  this is not  true as we can see in the  following example: we consider the right shift operator $R$ defined on $l^2.$ We  have
  $R$ is finite-isoloid  and $  \sigma_{d}(R)=D(0, 1);$ but $ \mbox{iso}\,\sigma(R)=\emptyset.$ Note that the mistake $(I_3)$  is originated in \cite{aiena2}.

  For $T\in L(X),$ we denote by $E_{f}(T)=\{\lambda \in \mbox{iso}\,\sigma(T): \alpha(T-\lambda I)<\infty\},$ and we say that $T$ is finitely if $\mbox{iso}\,\sigma(T)=E_{f}(T).$

 Now we give in the next proposition,  the  same version followed by a correct proof and under the weaker hypothesis that $T$ is finitely. We observe that if $T$ is finite-isoloid, then $T$ is finitely. Observe also that $\sigma_{b}(T)=\sigma_{d}(T)\Longleftrightarrow \Pi(T)=\Pi^{0}(T).$
 \begin{prop}
 If  $T\in L(X)$ is  a finitely operator,  then $\sigma_{b}(T) = \sigma_{d}(T).$
 \end{prop}
 \begin{proof}
 The inclusion $\sigma_{d}(T) \subset \sigma_{b}(T)$ is always true.   Let $\lambda \not\in \sigma_{d}(T)$ be arbitrary.  If $\lambda \not\in \sigma(T)$ then $\lambda \not\in \sigma_{b}(T).$ If $\lambda \in \sigma(T),$ then $\lambda \in \Pi(T)$ and so $\lambda \in \mbox{iso}\,\sigma(T).$ As $T$  is finitely  then $\lambda\in \Pi^{0}(T)$ and $\lambda \notin \sigma_{b}(T).$ Hence $\sigma_{b}(T) = \sigma_{d}(T).$

 For giving the  reader a good overview of the subject, we give another very  simple proof. Since $T$ is   finitely, then  $\Pi^{0}(T)=\Pi(T)\cap E_{f}(T)=\Pi(T).$
 \end{proof}

 \noindent {\bf \underline{Remark IV}:} In the paper \cite{jayanthi}, the authors defined a property  named $(Bv)$ as follows: $T\in L(X)$ satisfies property $(Bv)$ if $\sigma(T)\setminus\sigma_{uw}(T)=\sigma(T)\setminus\sigma_{ub}(T).$
   Moreover, they  studied this property in connection with different known Weyl type theorems and properties. But it is trivially that this property, is equivalent to the classical a-Browder's theorem; and all of the results announced in this paper are  already done.

\goodbreak
{\small \noindent Zakariae Aznay,\\  Laboratory (L.A.N.O), Department of Mathematics,\\Faculty of Science, Mohammed I University,\\  Oujda 60000 Morocco.\\
aznayweyl@gmail.com\\

\noindent Hassan  Zariouh,\newline Department of
Mathematics (CRMEFO),\newline
 \noindent and laboratory (L.A.N.O), Faculty of Science,\newline
  Mohammed I University, Oujda 60000 Morocco.\\
 \noindent h.zariouh@yahoo.fr

\end{document}